\documentclass[smallextended]{svjour3}[]

\usepackage{amsmath}
\usepackage{amsfonts, amssymb, euscript, mathrsfs}
\smartqed

\journalname{Annals of Finance}

\newcommand{\BR}{\mathbb{R}}

\newcommand{\SL}{\sum\limits}

\newcommand{\al}{\alpha}
\newcommand{\be}{\beta}
\newcommand{\ga}{\gamma}
\newcommand{\de}{\delta}
\newcommand{\De}{\Delta}

\newcommand{\CF}{\mathcal F}
\newcommand{\CB}{\mathcal B}

\newcommand{\MP}{\mathbf P}
\newcommand{\CM}{\mathcal M}

\newcommand{\Oa}{\Omega}
\newcommand{\si}{\sigma}

\renewcommand{\phi}{\varphi}
\newcommand{\vk}{\varkappa}
\newcommand{\eps}{\varepsilon}

\newcommand{\ol}{\overline}
\newcommand{\ta}{\theta}
\renewcommand{\d}{\mathrm{d}}

\begin{document}

\large

\title{\bf On a Class of Diverse Market Models}
\author{
Andrey Sarantsev
}  
\date{\today}
\institute{Andrey Sarantsev \at University of Washington \\ Department of Mathematics\\  Box 354350, Seattle, WA 98195-4350\\
\email{ansa1989@math.washington.edu}}
\authorrunning{Andrey Sarantsev}

\maketitle

\begin{abstract}
A market model in Stochastic Portfolio Theory is a finite system of strictly positive stochastic processes. Each process represents the capitalization of a certain stock. If at any time no stock dominates almost the entire market, which means that its share of total market capitalization is not very close to one, then the market is called diverse. There are several ways to outperform diverse markets and get an arbitrage opportunity, and this makes these markets interesting. A feature of real-world markets is that stocks with smaller capitalizations have larger drift coefficients. Some models, like the Volatility-Stabilized Model, try to capture this property, but they are not diverse. In an attempt to combine this feature with diversity, we construct a new class of market models. We find simple, easy-to-test sufficient conditions for them to be diverse and other sufficient conditions for them not to be diverse.
\end{abstract}


\section{Introduction}

Stochastic Portfolio Theory is a recently developed area of financial mathematics.  It is a flexible framework for analyzing portfolio behavior and equity market structure. See the book \cite{F2002} and the more recent survey \cite{FK2009} for detailed treatment of this topic.

Fix $n$, the number of stocks. Denote by $X_i(t)$ the total capitalization of the $i$th stock at time $t$. Let 
$$
S(t) = X_1(t) + \ldots + X_n(t)\ \ \mbox{and}\ \ \mu_i(t) = \frac{X_i(t)}{S(t)},\ \ i = 1, \ldots, n,\ \ t \ge 0,
$$
be the total capitalization of the market at time $t$ and the {\it market weights} of stocks, respectively. Fix a threshold $\de \in (0, 1)$. The market is called {\it $\delta$-diverse} if for every $t \ge 0$ and $i = 1, \ldots, n$ we have: 
$$
\mu_i(t) < 1 - \de.
$$
This definition was introduced in \cite{F2002}. Intuitively  it means that, at any given moment, no stock dominates almost the entire market. 

Stochastic Portfolio Theory studies market models that capture characteristics observed in the real-world markets. In particular, this theory is concerned with existence or absence of arbitrage opportunities (=portfolios outperforming the market) in these models. A few portfolios  outperforming the market were constructed in the articles \cite{FKK2005}, \cite{FK2005} and the survey \cite[Sections 7, 8, 11]{FK2009}. In these articles, it is proved that diverse markets can be outperformed. This is what makes diverse market models interesting. These markets were further studied in the articles \cite{FFR2007}, \cite{K2008}.

A few market models have recently attracted considerable attention. First, let us mention Geometric Brownian Motions with Rank-Dependent Drifts, see \cite{BFK2005}, \cite{PP2008}, \cite{CP2010}, \cite{IchibaThesis}, \cite{PS2010}, \cite{IKS2013}, \cite{FIK2013}. There are several generalizations of this model: second-order models, \cite{FIK2013b}, the hybrid Atlas models, \cite{Ichiba11}, competing Levy particles, \cite{S2011}, and particles with asymmetric collisions, \cite{KPS2012}. However, these models are not diverse. 

In real-world markets, stocks with lower market weights have larger drift and larger volatility. The Volatility-Stabilized Model, introduced in \cite{FKK2005} and further elaborated in \cite{P2011}, \cite{S2013}, \cite{BF2008}, attempts to capture this effect; see also the article \cite{P2013} for some generalizations. Unfortunately, it is also not diverse; see \cite[Theorem 3]{P2011}.

In this paper, we desire to combine diversity with this property of stocks with lower market weights. We introduce the following general class of models. Assume $g : (0, 1 - \de) \to \BR$ is a continuous function such that $\lim_{s \uparrow 1 - \de}g(s) = \infty$. Let $W(\cdot)$, where
$$
W(t) = ((W_1(t), \ldots, W_n(t)), t \ge 0),\ \ t \ge 0,
$$
be an $n$-dimensional standard Brownian motion. The market model is defined by the following system of stochastic differential equations: 
$$
\d\log X_i(t) = -g(\mu_i(t))\d t + \d W_i(t),\ \ i = 1, \ldots, n.
$$
We find conditions on the function $g$ which guarantee that the market is $\de$-diverse. We also find other conditions which guarantee that it is not $\de$-diverse.

In addition, we refer the reader to the following articles on this topic. Some diverse models were introduced in the article \cite{FK2005} and were also mentioned in the survey \cite[Chapter 7]{FK2009}. The paper \cite{OR2006} gives a measure-change method for constructing diverse market models. 

First, we study the case $n = 2$ (two stocks). The proof in this case is easier than in the general case, and so we are able to obtain necessary and sufficient conditions for diversity. Then, we consider the general case $n \ge 2$, in which we are only able to find some sufficient conditions for $\de$-diversity and some other sufficient conditions for the absence of $\de$-diversity. Our main technique is Feller's test for explosions, see for example \cite[Section 6.2]{DurBook}. 

The paper is organized as follows. In Section 2, we present basic definitions and statements of theorems, as well as concrete examples of functions $g$ for which the proposed market model is $\de$-diverse or not $\de$-diverse. In Section 3, we provide proofs for the case $n = 2$, and in Section 4, we consider the general case $n \ge 2$. The Appendix contains the formulation of Feller's test for explosions, as well as the statements and proofs of some auxillary lemmas.

\section{Main Results}

First, let us introduce the general definition of a market model. 
Consider a filtered probability space $(\Oa, \CF, (\CF_t)_{t \ge 0}, \MP)$, where the filtration $(\CF_t)_{t \ge 0}$ satisfies the usual conditions: it is right-continuous and augmented by $\MP$-null sets.
Let $n$ be the number of stocks. We denote the $n\times n$-identity matrix by $I_n$. Let $\de_{ij}$ stand for the Kronecker delta symbol: $\de_{ij} = 1$ for $i = j$ and $\de_{ij} = 0$ for $i \ne j$. 

Fix an integer $d \ge 1$. Let $W(\cdot)$, where
$$
W(t) = (W_1(t), \ldots, W_d(t)),\ \ t \ge 0,
$$
be a $d$-dimensional standard Brownian motion. Assume the filtration $(\CF_t)_{t \ge 0}$ is generated by this Brownian motion. Assume $\ga(\cdot)$, where 
$$
\ga(t) = (\ga_1(t), \ldots, \ga_n(t)), \ \forall t \ge 0,
$$
is an $\BR^n$-valued progressively measurable stochastic process, and $\si(\cdot)$, where
$$
\si(t) = (\si_{ij}(t))_{1 \le i \le n,1 \le j \le d}, \ \forall t \ge 0
$$
is a matrix-valued $(\CF_t)_{t \ge 0}$-progressively measurable stochastic process. The size of this matrix is $n\times d$. We impose a technical condition: for every finite $T \ge 0$, each $\ga_i$ is integrable and each $\si_{ij}$ is square-integrable on $[0, T]$ a.s. 

\begin{definition}
Assume that an $n$-dimensional process $X(\cdot)$, where 
$$
X(t) = (X_1(t), \ldots, X_n(t)),\ \ t \ge 0,
$$
with the state space $(0, \infty)^n$, satisfies the following system of stochastic differential equations:
$$
\d\log X_i(t) = \ga_i(t)\d t + \SL_{j=1}^d\si_{ij}(t)\d W_j(t),\ \ i = 1, \ldots, n.
$$
Then this process $X(\cdot)$ is called the {\it market model of $n$ stocks} with {\it growth rates} $\ga_i(t)$, $i = 1, \ldots, n$, and {\it volatility matrix} $\si(t)$. The value $X_i(t)$ is called the {\it capitalization of the $i$th stock at time $t \ge 0$}. 
\end{definition}

We have already defined market weights and the notion of diversity in the introduction. Loosely speaking, a market is diverse if at every moment no stock dominates almost the entire market.  Let us recall this definition:

\begin{definition}
\label{defnweights}
Denote
\begin{equation}
\label{weights}
S(t) = X_1(t) + \ldots + X_n(t)\ \ \mbox{and} \ \ \mu_i(t): = \frac{X_i(t)}{S(t)},\ \ i = 1, \ldots, n.
\end{equation}
The quantity $S(t)$ is called {\it the total market capitalization} at time $t$. The quantity $\mu_i(t)$ is called {\it the market weight} of the $i$th stock at time $t$. The {\it market weights vector} is defined by
$$
\mu(t) \equiv (\mu_1(t), \ldots, \mu_n(t)),\ \ t \ge 0.
$$ 
\end{definition} 

\begin{definition} The market $X(\cdot)$, where $X(t) = (X_1(t), \ldots, X_n(t))$, $t \ge 0$, is {\it diverse with parameter $\de > 0$}, or simply {\it $\de$-diverse}, if a.s. for each $i = 1, \ldots, n$ and every $ t \ge 0$ we have: $\mu_i(t) < 1 - \de$. 
The parameter $\de$ is called a {\it diversity threshold}. 
\end{definition}

This can be viewed as a consequence of an antitrust legislation. See, for example, an interesting paper \cite{FS2011}. Incidentally, note the strict inequality $\mu_i(t) < 1 - \de$ instead of $\mu_1(t) \le 1 - \de$ in \cite{F2002} and \cite{FK2009}. The definition with $\le$ is common in the literature, but the definition with strict inequality is more convenient in our proofs.  

Let us introduce a new class of market models. We fix a parameter $\de \in (0, 1)$, which plays the role of the diversity threshold. 

\begin{definition} Assume a continuous function $g : (0, 1 - \de) \to \BR$ satisfies the following condition: $\lim_{x \uparrow 1 - \de}g(x) = \infty$.
Then the function $g(\cdot)$ is called {\it admissible}. 
\end{definition}

For every admissible $g(\cdot)$, consider the market model defined by: 
$$
n = d,\ \ga_i(t) = -g(\mu_i(t)),\ \si(t) = I_n,\ t \ge 0.
$$
In other words, consider the vector-valued stochastic process $X(\cdot)$, where 
$$
X(t) \equiv (X_1(t), \ldots, X_n(t)),\ t \ge 0,
$$
satisfying the following system of stochastic differential equations:
\begin{equation}
\label{main}
\d\log X_i(t) = -g(\mu_i(t))\d t + \d W_i(t),\ \ i = 1, \ldots, n,
\end{equation}
where $W(\cdot) = (W_1(t), \ldots, W_n(t), t \ge 0)$ is an $n$-dimensional Brownian motion. 

\begin{lemma}
\label{majorlemma}
Assume the market model is governed by~\eqref{main} for some admissible function $g$. Let us define an auxillary function:
\begin{equation}
\label{psi}
\psi(s) := s(-g(s) + 1/2) - s^2,\ \ s \in (0, 1 - \de).
\end{equation}
Then the market weights, defined in~\eqref{weights}, satisfy the following system of stochastic differential equations:
\begin{equation}
\label{weights}
\d\mu_i(t) = \biggl(\psi(\mu_i(t)) - \mu_i(t)\SL_{j=1}^n\psi(\mu_j(t))\biggr)\d t + \SL_{j=1}^n(\de_{ij}\mu_i(t) - \mu_i(t)\mu_j(t))\d W_j(t),
\end{equation}
 for $i = 1, \ldots, n,\ t \ge 0$. 
\end{lemma}

A fairly straightforward proof (by It\^o's lemma) is postponed until the Appendix. Now we can state our main results. 
Define
\begin{equation}
\label{A1A2}
A_1(x) := \frac2{1 + (n-1)^{-1}}\frac1{x(1 - x)},\ \ A_2(x) := \frac1{x(1 - x)},\ \ a_1 = A_1(1 - \de),\ \ a_2 = A_2(1 - \de).
\end{equation}
First, consider the case $n = 2$ (two stocks). Here, the threshold $\de$, defined in Definition~\ref{defnweights}, must be between $0$ and $1/2$. Indeed, $\mu_1(t) + \mu_2(t) = 1,\ t \ge 0$, so at least one of market weights $\mu_1(t)$ and $\mu_2(t)$ must be greater than or equal to $1/2$. 

\begin{theorem} 
\label{thm1}
Consider the case of two stocks, $n = 2$. Fix the parameter $\de \in (0, 1/2)$. Assume the market model is governed by the equations~\eqref{main} for some admissible function $g(\cdot)$. 
The market is $\de$-diverse if and only if 
\begin{equation}
\label{mainthm}
\int_{1/2}^{1 - \de}\exp\left(\int_{1/2}^yA_2(z)g(z)dz\right)dy = \infty.
\end{equation}
\end{theorem}

\begin{corollary} Consider the case of two stocks, $n = 2$. Fix the parameter $\de \in (0, 1/2)$. Assume the market model is given by~\eqref{main}.
\label{cor1}

(i) If $\int_{1/2}^{1 - \de}g(z)dz = \infty$ and for some $\eps > 0$ we have:
\begin{equation}
\label{simple}
\int_{1/2}^{1 - \de}\exp\left((a_2 - \eps)\int_{1/2}^yg(z)dz\right)dy = \infty,
\end{equation}
then the market is $\de$-diverse.

(ii) If $\int_{1/2}^{1 - \de}g(z)dz = \infty$, but
$$
\int_{1/2}^{1 - \de}\exp\left(a_2\int_{1/2}^yg(z)dz\right)dy < \infty,
$$
then the market is not $\de$-diverse.

(iii) If $\int_{1/2}^{1 - \de}g(z)dz < \infty$, then the market is not $\de$-diverse. 
\end{corollary}

Now, consider the general case: $n \ge 2$, the number of stocks is any fixed number greater than or equal to two.

\begin{theorem}
\label{thm2}
Fix $\de \in (0, 1/2)$. Assume the market model is governed by the equations~\eqref{main} for some admissible $g$ which satisfies
\begin{equation}
\label{boundedness}
-\infty < \varliminf\limits_{x \downarrow 0}xg(x) \le \varlimsup\limits_{x \downarrow 0}xg(x) < 0.
\end{equation}
Fix some $x_0 \in (0, 1 - \de)$. 

(i) If
$$
\int_{x_0}^{1 - \de}\exp\left(\int_{x_0}^yA_1(z)g(z)dz\right)dy < \infty,
$$
then the market is not $\de$-diverse. 

(ii) If 
$$
\int_{x_0}^{1 - \de}\exp\left(\int_{x_0}^yA_2(z)g(z)dz\right)dy = \infty,
$$
then the market is $\de$-diverse. 
\end{theorem}

\begin{corollary} Fix the parameter $\de \in (0, 1/2)$. Take any point $x_0 \in (0, 1 - \de)$. Assume the market model is governed by the equations~\eqref{main} for some admissible function $g$    which satisfies the condition~\eqref{boundedness}. 

(i) If $\int_{x_0}^{1 - \de}g(z)dz = \infty$ and for some $\eps > 0$ we have: 
$$
\int_{x_0}^{1 - \de}\exp\left((a_2 - \eps)\int_{x_0}^yg(z)dz\right)dy = \infty,
$$
then the market is $\de$-diverse. 

(ii) If $\int_{x_0}^{1 - \de}g(z)dz = \infty$ and 
$$
\int_{x_0}^{1 - \de}\exp\left(a_1\int_{x_0}^yg(z)dz\right)dy  < \infty,
$$
then the market is not $\de$-diverse.

(iii) If $\int_{x_0}^{1 - \de}g(z)dz < \infty$, then the market is not $\de$-diverse.

\label{cor2}
\end{corollary}

The proofs of Theorems 1 and 2 are given in Sections 3 and 4.  In both cases, to prove that the market is $\de$-diverse, we must show that no market weight hits zero or $1 - \de$. Indeed, if a certain market weight hits zero, then the respective capitalization $X_i(t)$ becomes zero, and this is forbidden by the definition of the market model. If a certain market weight hits $1 - \de$, then the market model is not $\de$-diverse. Similarly, to prove that the market is not $\de$-diverse, we must show that no market weight hits zero, but at least one market weight hits the point $1 - \de$. 

In the case $n = 2$, we can express $\mu_2(t) = 1 - \mu_1(t)$ and write a stochastic differential equation for $\mu_1(t)$. Then we apply Feller's explosion test from \cite[Section 6.2]{DurBook} for $\mu_1$. For convenience, we provide the general statement of this test in the Appendix. For the general case, we take each market weight and estimate its drift and diffusion, to find whether it hits or does not hit $0$ and $1 - \de$. In this case, the proof is also based on Feller's test, but we were not able to find  conditions which are both necessary and sufficient. We only found a condition under which the market is $\de$-diverse, and another condition under which it is not $\de$-diverse. These two conditions are pretty close to each other, but there is still a gap between them.

Note that, in the proposed model, the drift is larger for smaller stocks, but the volatility coefficient is the same (unit) for all stocks. However, in volatility-stabilized models, as well as in real-world markets, the volatility coefficient has this same property as the drift coefficient: it is larger for stocks with smaller capitalization. To capture this, it would be nice to introduce the model
$$
\d\log X_i(t) = g(\mu_i(t))\d t + h(\mu_i(t))\ \d W_i(t),\ \ t \ge 0,\ \ i = 1, \ldots, n.
$$
Here, $g, h : (0, 1 - \de) \to \BR$ are continuous functions. It is worth to find conditions on the functions $g$ and $h$ which guarantee existence or absence of $\de$-diversity. Alternatively, consider even more general model
$$
\d\log X_i(t) = g_i(\mu_i(t))\d t + h_i(\mu_i(t))\ \d W_i(t),\ \ t \ge 0,\ \ i = 1, \ldots, n,
$$
where the functions $g_i, h_i : (0, 1 - \de) \to \BR,\ i = 1, \ldots, n$, are continuous. This is an interesting topic for future research.

Let us return to the proposed class of models and consider some of examples. 

\begin{example} Assume there are two stocks, $n = 2$. Check the conditions of Corollary~\ref{cor1} for the function
$$
g(y) = \frac{p}{1 - \de - y},\ \ \mbox{where}\ \ p > 0.
$$
Then $\int_{1/2}^{1 - \de}g(y)dy = \infty$. For $\al > 0$ we have:
$$
\exp\left(\al\int_{1/2}^xg(y)dy\right) = \exp\left(-\al p\log(1 - \de - x) + \al p\log(1/2 - \de)\right) = c(1 - \de - x)^{-\al p},
$$
where $c > 0$ is a certain constant (dependent on $\al$ and $p$ but not on $x$). Therefore, 
$$ 
\int_{1/2}^{1 - \de}\exp\left(\al\int_{1/2}^xg(y)dy\right)dx < \infty\ \ \mbox{if and only if} \ \ \al p < 1.
$$
If $p < \de(1 - \de)$, then take $\al = a_2 = (\de(1 - \de))^{-1}$  and get: $\al p < 1$, so the market is not $\de$-diverse. If $p > \de(1 - \de)$, then take $\al = a_2 - \eps = (\de(1 - \de))^{-1} - \eps$ for sufficiently small $\eps > 0$ and get: $\al p \ge 1$, so the market is $\de$-diverse. For $p = \de(1 - \de)$, Corollary~\ref{cor1} does not work. We must use the general Theorem~\ref{thm1}. We can verify directly that the condition~\eqref{mainthm} is true for this function $g$. Thus, the market is $\de$-diverse. 
\end{example}

\begin{example} Consider again the case $n = 2$, two stocks. Check the conditions of Corollary~\ref{cor1} for the function
$$
g(y) = \frac{p}{(1 - \de - y)^q},\ \ \mbox{where}\ \ p, q > 0.
$$
We have just discussed the case $q = 1$. 
For $q < 1$, we have: $\int_{1/2}^{1 - \de}g(y)dy < \infty$, so the market is not $\de$-diverse. For $q > 1$, we have: 
$$
\int_{1/2}^yg(z)dz = \frac{p}{q(1 - \de - y)^{q-1}} - c,
$$
 where $c$ is a certain constant, dependent on $p, q$, but not $x$. Therefore, $\int_{1/2}^{1 - \de}g(z)dz = \infty$. In addition, it is easy to see that for every constant $\al > 0$, we have:
\begin{equation}
\int_{1/2}^{1 - \de}\exp\left(\al\int_{1/2}^yg(z)dz\right)dy = \infty.
\end{equation}
Therefore, the market is $\de$-diverse. 
\end{example}

\begin{example} Now, consider the general case, $n \ge 2$. Take an admissible function $g$ which satisfies the condition~\eqref{boundedness}. Assume that in some left neighborhood $(x_0, 1 - \de)$ of $1 - \de$, where $x_0 \in (0, 1- \de)$ is a fixed number, it is given by the formula
$$
g(z) = \frac{p}{1 - \de - z}.
$$
Here, similarly to Example 1, we have: from Corollary 2, if $p < 1/a_1$, then the market is not $\de$-diverse, and if $p > 1/a_2$, then the market is $\de$-diverse. If $p = 1/a_2$, then we must use Theorem 2 instead of Corollary 2, which gives us that the market is $\de$-diverse. If $1/a_1 \le p < 1/a_2$, then the question remains open.
\end{example}

\begin{example} Make the same assumptions as in the previous example, except that, in some left nieghborhood of $1 - \de$, we have:
$$
g(z) = \frac{p}{(1 - \de - z)^q},\ \ \mbox{where}\ \ p, q > 0.
$$
Then, similarly to Example 2, we have: for $q < 1$, this market is not $\de$-diverse, and for $q > 1$, it is always diverse. The case $q = 1$ was already considered. 
\end{example}

\section{Proof of Theorem~\ref{thm1} and Corollary~\ref{cor1}}

{\it Proof of Theorem~\ref{thm1}.} Since there are only two stocks, we have: $\mu_2(t) = 1 - \mu_1(t)$. To prove $\de$-diversity, it suffices to show that $\mu_1$ does not hit $1 - \de$. Indeed, these two stocks are absolutely symmetric, and the proof that $\mu_2$ does not hit $1 - \de$ is repeated verbatim. After we prove that both of these weights are strictly less than $1 - \de$, we can immediately conclude that both are strictly greater than $\de$, since they add up to $1$. Therefore, they do not hit zero. Thus, we arrive at the conclusion that this is a $\de$-diverse market model.

Plug $\mu_2(t) = 1 - \mu_1(t)$ into the equation~\eqref{weights}. The market weight $\mu_1$ satisfies the equation:
\begin{align*}
\d\mu_1(t) = & \Bigl[\psi(\mu_1(t))(1 - \mu_1(t)) - \mu_1\psi(1 - \mu_1(t))\Bigr]\d t  \\ & + (\mu_1(t) - \mu_1^2(t))\d W_1(t) - (\mu_1(t) - \mu_1^2(t))\d W_2(t),\ \ t \ge 0.
\end{align*}
Define the process $B(\cdot) =  (B(t), t \ge 0)$ by 
$$
B(t) := \frac1{\sqrt{2}}(W_1(t) - W_2(t)).
$$
This is a one-dimensional standard Brownian motion. Therefore, $\mu_1(\cdot)$ satisfies the following stochastic differential equation:
$$
\d\mu_1(t) = b_0(\mu_1(t))\d t + \si(\mu_1(t))\d B(t),
$$
where the drift and diffusion coefficients $b_0$ and $\si$ are defined as:
$$
b_0(x) := (1-x)\psi(x) - x\psi(1 - x),\ \  \mbox{and}\ \ \si(x) := \sqrt{2}(x - x^2).
$$
Use Feller's test for explosions, see the Appendix. We apply this test for the interval $(\de, 1 - \de)$ of values of $\mu_1$. Let us follow the notation from Appendix. Take $x_0 = 1/2$. The natural scale is defined as
$$
\phi(x) = \int_{1/2}^x\exp\left(\int_{1/2}^y-\frac{2b_0(z)}{\si^2(z)}dz\right)dy.
$$
Let $m(x) := 1/(\phi'(x)\si^2(x))$. The following conditions are equivalent: $\mu_1$ does not hit $1 - \de$ a.s. if and only if either $\phi(1 - \de) = \infty$, or $\phi(1 - \de) < \infty$ but 
$$
\int_{1/2}^{1 - \de}m(x)(\phi(1 - \de) - \phi(x))dx = \infty.
$$
Now, let us plug $b_0$ and $\si$ into these formulas:
\begin{align*}
\frac{2b_0(z)}{\si^2(z)} =& \frac{2((1-z)\psi(z) - z\psi(1 - z))}{2z^2(1 - z)^2} = \frac{\psi(z)}{z^2(1 - z)} - \frac{\psi(1 - z)}{z(1 - z)^2}  
= \frac{z(-g(z) + 1/2) - z^2}{z^2(1 - z)} \\ &- \frac{(1 - z)(-g(1 - z) + 1/2) - (1 - z)^2}{z(1 - z)^2}  
= -\frac{g(z)}{z(1 - z)} + \frac{g(1 - z)}{z(1 - z)} + \frac{1 - 2z}{z(1 - z)} \\ &= -\frac{g(z)}{z(1 - z)} + \eps_1(z),
\end{align*}
where we let
$$
\eps_1(z) := \frac{g(1 - z) + 1 - 2z}{z(1 - z)}.
$$
Since the function $g(\cdot)$ is continuous on $[\de, 1/2]$, there exists a constant $C_1 > 0$ such that for every $z \in [1/2, 1 - \de]$ we have: $|\eps_1(z)| \le C_1$. Let us continue our calculations:  
\begin{multline*}
G(y) := \exp\left(\int_{1/2}^y-\frac{2b_0(z)}{\si^2(z)}dz\right)  = \exp\left(\int_{1/2}^y\left(\frac{g(z)}{z(1 - z)} - \eps_1(z)\right)dz\right) = e^{F(y)}\eps_2(y),
\end{multline*}
where for $y \in [1/2, 1 - \de)$ we have:
$$
F(y) := \int_{1/2}^y\frac{g(z)}{z(1-z)}dz = \int_{1/2}^yA_2(z)g(z)dz,\ \ \eps_2(y) := \exp\left(-\int_{1/2}^{y}\eps_1(z)dz\right).
$$
Let $C_2 := \exp((1/2 - \de)C_1)$. Then for every $y \in [1/2, 1 - \de)$ we have:
$$
0 < C_2^{-1} \le \eps_2(y)  = \frac{G(y)}{e^{F(y)}} \le C_2 < \infty.
$$
Therefore, the integrals
$$
\phi(1 - \de) = \int_{1/2}^{1 - \de}G(y)dy \ \ \mbox{and}\ \ \int_{1/2}^{1 - \de}e^{F(y)}dy
$$
either both converge or both diverge. If they diverge, then the process $\mu_1(\cdot)$ does not hit $1 - \de$, and the market is $\de$-diverse. If they both converge, then let us show that
$$
\int_{1/2}^{1 - \de}(\phi(1 - \de) - \phi(x))m(x)dx < \infty.
$$
This will prove that $\mu_1(\cdot)$ does hit $1 - \de$ with positive probability, so the market is not $\de$-diverse. Indeed, 
$$
\phi(1 - \de) - \phi(x) = \int_x^{1 - \de}G(y)dy, 
$$
and, in addition,
$$
 m(x) := \frac1{\phi'(x)\si^2(x)} = \frac1{G(x)}\cdot\frac1{2x^2(1 - x)^2}.
$$
Therefore, we have:
\begin{equation}
\label{firstint}
\int_{1/2}^{1 - \de}(\phi(1 - \de) - \phi(x))m(x)dx = \int_{1/2}^{1 - \de}\int_x^{1 - \de}G(y)dy\cdot
\frac{1}{G(x)}\cdot\frac{dx}{2x^2(1 - x)^2}.
\end{equation}
Note that  the function $1/(2x^2(1 - x)^2)$ is bounded from above for $x \in [1/2, 1- \de)$. To show that the integral from~\eqref{firstint} is finite, we need to prove that the integral
\begin{equation}
\label{secondint}
\int_{1/2}^{1 - \de}\left(\int_x^{1 - \de}e^{F(y)}dy\right)\frac{dx}{e^{F(x)}}
\end{equation}
is finite. Recall that, according to Definition 4, $g(x) \to \infty$ as $x \uparrow 1 - \de$. There exists $x_0 \in (1/2, 1 - \de)$ such that $g(x) \ge 0$ for $x \in [x_0, 1 - \de)$. The function $F(\cdot)$ increases on $[x_0, 1 - \de)$, and so does $e^{F(\cdot)}$. Therefore, $e^{F(x)} \ge e^{F(x_0)}$ for $x \in [x_0, 1 - \de)$. Since $F(\cdot)$ is continuous on $[1/2, x_0]$, it is bounded from below on this interval, and so is $e^{F(\cdot)}$. We conclude that the function $e^{F(\cdot)}$ is bounded from below on $[1/2, 1 - \de)$ by some positive constant $C_3 > 0$. Therefore, we have:
\begin{align*}
\int_{1/2}^{1 - \de}\left(\int_x^{1 - \de}e^{F(y)}dy\right)\frac{dx}{e^{F(x)}}  &\le \frac1{C_3}\int_{1/2}^{1 - \de}\left(\int_{1/2}^{1 - \de}e^{F(y)}dy\right)dx \\ &= \frac{(1 - \de) - 1/2}{C_3}\int_{1/2}^{1 - \de}e^{F(y)}dy < \infty.\ \ \blacksquare
\end{align*}

{\it Proof of Corollary~\ref{cor1}.} (i) Let us show that
\begin{equation}
\label{infty}
\int_{1/2}^{1 - \de}\exp\left(\int_{1/2}^yA_2(z)g(z)dz\right)dy = \infty.
\end{equation}
Since the function $A_2(\cdot)$ is continuous and increasing on $[1/2, 1 - \de]$, there exists $x_0 \in (1/2, 1 - \de)$ such that for $z \in [x_0, 1 - \de]$ we have: $a_2 - \eps \le A_2(z) \le a_2$. Since $g(1 - \de) = \infty$, without loss of generality (increasing the value of $x_0$ if necessary) we can assume $g(z) \ge 0$ for all $z \in [x_0, 1 - \de)$. To prove~\eqref{infty}, let us show that
\begin{equation}
\label{second}
\int_{x_0}^{1 - \de}\exp\left(\int_{1/2}^yA_2(z)g(z)dz\right)dy = \infty.
\end{equation}
The expression under the outer integral can be rewritten as
$$
\exp\left(\int_{x_0}^{y}A_2(z)g(z)dz\right) \cdot \exp\left(\int_{1/2}^{x_0}A_2(z)g(z)dz\right).
$$
The second factor does not depend on $y$. Therefore, to show~\eqref{second}, it suffices to prove that
$$
\int_{x_0}^{1 - \de}\exp\left(\int_{x_0}^{y}A_2(z)g(z)dz\right)dy = \infty.
$$
Since $A_2(z) \ge a_2 - \eps$ and $g(z) \ge 0$ for $z \in [x_0, 1 - \de)$, we have:
$$
\int_{x_0}^{1 - \de}\exp\left(\int_{x_0}^{y}A_2(z)g(z)dz\right)dy \ge \int_{x_0}^{1 - \de}\exp\left((a_2 - \eps)\int_{1/2}^yg(z)dz\right)dy = +\infty.
$$
The proof of (i) is complete. Part (ii) is proved analogously. Let us show (iii). Assume 
$\int_{1/2}^{1 - \de}g(z)dz < \infty$. The function $A_2(\cdot)$ is bounded on $[1/2, 1 - \de]$. Therefore, for $y \in [1/2, 1 - \de)$ the function $F(\cdot)$ defined in the proof of Theorem~\ref{thm1} is bounded from above. Thus, the function $e^{F(\cdot)}$ is also bounded on $[1/2, 1 - \de)$, and therefore it is integrable on this interval.   $\blacksquare$

\section{Proof of Theorem~\ref{thm2} and Corollary~\ref{cor2}}

{\it Proof of Theorem~\ref{thm2}.} 
Let us first informally present the idea of the proof. 

{\bf Step 1:} Choose a market weight, say $\mu_1$, and find whether it hits $0$ or $1 - \de$. Write an equation for $\mu_1$ in the form:
\begin{equation}
\label{ito}
\d\mu_1(t) = \be(t)\d t + \rho(t)\d B(t).
\end{equation}
Here, $\be = (\be(t), t \ge 0)$ and $\rho = (\rho(t), t \ge 0)$ are certain random processes, and $B = (B(t), t \ge 0)$ is a one-dimensional standard Brownian motion. 

{\bf Step 2:}  The first difficulty is that $\be(t)$ and $\rho(t)$ are {\it not} functions of $\mu_1(t)$. As we shall see in the sequel, these are functions of the {\it whole} market weights vector $\mu(t) \equiv (\mu_1(t), \ldots, \mu_n(t))$. In the case of two stocks ($n = 2$), we could express the other market weight $\mu_2(t)$ as $1 - \mu_1(t)$, so $\be(t)$ and $\rho(t)$ were functions only of $\mu_1(t)$. Therefore, $\mu_1$ satisfied a certain stochastic differential equation, for which we could just apply Feller's test in the most straightforward manner. Here, we can only find some lower and upper estimates for $\be(t)$ and $\rho(t)$ of the form:
\begin{equation}
\label{estimates}
b_1(\mu_1(t)) \le \be(t) \le b_2(\mu_1(t)) \ \ \mbox{and}\ \ \si_1(\mu_1(t)) \le \rho(t) \le \si_2(\mu_1(t)).
\end{equation}

{\bf Step 3:} Having found these estimates, one might want to compare $\mu_1$ to solutions of SDEs.
However, this approach does not work directly. When we compare two It\^o processes, they must have the same diffusion coefficient, as in \cite[Proposition~5.2.18]{KS1991}. We resolve this difficulty by making the diffusion coefficients equal to each other using an appropriate time-change, as in \cite[Section 3]{H1985}. On the general theory of time-change, see the book \cite[Section 3.4.B]{KS1991}. Then we use Feller's test to find whether the solutions of these stochastic differential equations hit or do not hit $0$ and $1 - \de$. We shall prove that, under the conditions of this theorem, they do not hit $0$.

Now, let us carry out the proof of Theorem 2 in full detail, following the steps outlined above.

\medskip

{\bf Step 1:} Choose one of the market weights, say $\mu_1$. It satisfies the equation 
\begin{align*}
\d\mu_1(t) =& \biggl[\psi(\mu_1(t)) - \mu_1(t)\SL_{j=1}^n\psi(\mu_j(t))\biggr]\d t + \mu_1(t)\d W_1(t) - \mu_1(t)\SL_{j = 1}^n\mu_{j}(t)\d W_j(t)  \\ &= 
\biggl[\psi(\mu_1(t))(1 - \mu_1)  - \mu_1(t)\SL_{j=2}^n\psi(\mu_j(t))\biggr]\d t
\\ &+ \left(\mu_1(t) - \mu_1^2(t)\right)\d W_1(t) - \mu_1(t)\SL_{j=2}^n\mu_j(t)\, \d W_j(t),\ \ t \ge 0.
\end{align*}
Rewrite this equation as
$$
\d\mu_1(t) = \be(t)\d t + \rho(t)\d B(t).
$$
Here $B = (B(t), t \ge 0)$ is a one-dimensional standard Brownian motion, and
$$
\be(t) = \psi(\mu_1(t))(1 - \mu_1(t)) - \mu_1(t)\sum_{j =2}^n\psi(\mu_j(t))
$$
and
$$
\rho(t) = \left[(\mu_1(t) - \mu_1^2(t))^2 + \mu_1^2(t)\sum_{j=2}^n\mu^2_j(t)\right]^{1/2}
$$
are drift and diffusion coefficients. As mentioned above, they depend on the whole market weights vector $\mu(t)$, and not just on $\mu_1(t)$. 

\medskip

{\bf Step 2:} Let us find lower and upper estimates for $\be(t)$ and $\rho(t)$ which depend only on $\mu_1(t)$.  

(a) First, consider $\rho(t)$. We have: 
$$
\SL_{j = 2}^n\mu_j(t) = 1 - \mu_1(t),\ \ \mbox{and}\ \ \mu_j(t) > 0 \ \ \mbox{for}\ \ j = 2, \ldots, n.
$$
Fix a positive integer $m$ and a real number $a > 0$. Consider the expression
$$
y_1^2 + \ldots + y_m^2,\ \ \mbox{where}\ \ y_1, \ldots, y_m \ge  0\ \ \mbox{and} \ \ y_1 + \ldots + y_m = a.
$$
Its maximal value is achieved when one of the variables $y_i$ is equal to $a$ and all others are $0$, and this maximal value is equal to $a^2$. Its minimal value is achieved when all variables are equal to each other (i.e. equal to $a/m$), and this minimal value is equal to $m(a/m)^2 = a^2/m$. 
Therefore, 
$$
(n-1)^{-1}(1 - \mu_1(t))^2 \le \sum_{j = 2}^n\mu_j^2(t) \le (1 - \mu_1(t))^2.
$$
Plugging this into the formula for $\rho(t)$, we get the following estimates:
$$
\Bigl[\mu_1^2(1 - \mu_1(t))^2(1 + (n-1)^{-1})\Bigr]^{1/2} \le \rho(t) \le \Bigl[2\mu_1^2(t)(1 - \mu_1(t))^2\Bigr]^{1/2}.
$$
Rewrite this as
\begin{equation}
\label{sigma}
\vk\si(\mu_1(t)) \le \rho(t) \le \si(\mu_1(t)),
\end{equation}
where, as in the previous section, $\si(x) = \sqrt{2}x(1 - x)$, and 
$$
\vk := \left(\frac{1 + (n-1)^{-1}}2\right)^{1/2}
$$
is a constant, $0 < \vk \le 1$. 

\medskip

(b) Let us find similar estimates for the drift $\be(t)$. 
Condition~\eqref{boundedness}, imposed in Theorem~\ref{thm2}, implies that the function 
$\psi(x) := x(1/2 - g(x)) - x^2$ satisfies 
$$
0 < \varliminf\limits_{x \downarrow 0}\psi(x) \le \varliminf\limits_{x \downarrow 0}\psi(x) < \infty.
$$
Therefore, $\psi(\cdot)$ is bounded on $(0, 1/2]$. 
Also, from condition $\lim_{x \uparrow 1 - \de}g(x) = +\infty$ we imply that $\lim_{x \uparrow 1 - \de}\psi(x) = -\infty$, so $\psi(\cdot)$ is bounded from above on $(0, 1 - \de)$. Therefore, there exists a constant $C_1 > 0$ such that for every vector $(y_1, \ldots, y_{n-1})$ which satisfies the following properties: 
$$
0 < y_i < 1 - \de,\ \ i = 1, \ldots, n-1,\ \ y_1 + \ldots + y_{n-1} < 1,
$$
the following estimate holds true:
\begin{equation}
\label{upperC1}
\SL_{j=1}^{n-1}\psi(y_j) \le C_1.
\end{equation}
Therefore, we have: $\sum_{j=2}^{n}\psi(\mu_j(t)) \le C_1$. Thus, 
\begin{equation}
\label{lowerdrift}
\be(t) \ge b_1(\mu_1(t)),\ \ \mbox{where}\ \ b_1(x) := \psi(x)(1 - x) - C_1x.
\end{equation}
Let us also find an upper estimate. For any fixed $x \in (0, 1 - \de)$, let 
\begin{equation}
\label{theta}
\ta(x) := \inf\SL_{j=1}^{n-1}\psi(y_j),
\end{equation}
where the minimum is taken over all vectors $(y_1, \ldots, y_{n-1}) \in \BR^{n-1}$ such that 
$$
0 < y_i < 1 - \de,\ \ i = 1, \ldots, n - 1,\ \ y_1 + \ldots + y_{n-1} < 1 - x.
$$
This infinum is finite, since the function $\psi(\cdot)$ is bounded on $(0, 1 - x)$.  Let $b_2(x) := \psi(x)(1 - x) - \ta(x)x$. 
We have:  
\begin{equation}
\label{upperdrift}
\be(t) \le b_2(\mu_1(t)),\ \ \mbox{where}\ \ b_2(x) := \psi(x)(1 - x) - \ta(x)x.
\end{equation}
Let us mention some properties of the function $\ta(\cdot)$ and $b_2(\cdot)$. Since $\de < 1/2$, we have: $\de < 1 - \de$. When $x  \in [1/2, 1 - \de)$, we have: $1 - x \in (\de, 1/2]$. As mentioned before, the function $\psi(\cdot)$ is bounded on $(0, 1/2]$. It easily follows from the definition of $\ta(x)$ that this function is bounded for $x \in (1/2, 1]$. 
Recall the assumptions of Theorem~\ref{thm2}: $\varlimsup_{x \downarrow 0}xg(x) < 0$. Therefore, $\varliminf_{x \downarrow 0}\psi(x) > 0$, and we have: $\varliminf_{x \downarrow 0}b_1(x) > 0$. Hence, there exist $K_1 > 0$ and $x_1 \in (0, 1 - \de)$ such that for every $x \in (0, x_1]$, we have: $b_1(x) \ge K_1$.
Let $C_2$ be a constant such that $\ta(x) \ge C_2$ for $x \in (1/2, 1-\de]$. Then, for $x \in (1/2, 1 - \de)$ we have:   
$$
b_2(x) \le \ol{b}_2(x) := \psi(x)(1 - x) - C_2x.
$$
Moreover, $\lim_{x \uparrow 1 - \de}\psi(x) = -\infty$, and so $\lim_{x \uparrow 1 - \de}\ol{b}_2(x) = -\infty$. There exists $x_2 \in [1/2, 1 - \de)$ such that for $x \in [x_2, 1 - \de)$ we have:
$$
b_2(x) \le \ol{b}_2(x) < 0.
$$

\medskip

{\bf Step 3:}  Let us carry out the time-change mentioned above. Let 
$$
\Delta(t) = \int_0^t\frac{\rho^2(u)}{\si^2(\mu_1(u))}du,\ \ t \ge 0.
$$
From~\eqref{sigma} it follows that $\vk^2 \le \Delta'(t) \le 1$ for all $t \ge 0$. Therefore, $\Delta(\cdot)$ is a strictly increasing, continuously differentiable function with $\Delta(0) = 0$ and $\Delta(\infty) = \infty$. Let $\tau(\cdot)$ be its inverse function:
$$
\tau(s) = \inf\{t \ge 0\mid \Delta(t) \ge s\}.
$$
Let $Z(t) = \mu_1(\tau(t))$. By Lemma~\ref{Hajek} (see Appendix), 
$$
\d Z(s) = \be(\tau(s))\frac{\sigma^2(Z(s))}{\rho^2(\tau(s))}\d s + \sigma(Z(s))\d B_0(s),
$$
where $B_0 = (B_0(s), s \ge 0)$ is an $(\CF_{\Delta(t)})_{t \ge 0}$-Brownian motion.

Our goal is to compare the process $Z(\cdot) = (Z(t), t \ge 0)$ with solutions of one-dimensional stochastic differential equations to find whether this process hits or does not hit $0$ and $1 - \de$. The range of $\Delta(t)$ is $[0, \infty)$, so the process $Z(\cdot)$ hits a certain point if and only if the market weight $\mu_1(\cdot)$ hits this point. 
 
Recall the estimates from Step 2:
$$
b_1(\mu_1(t)) \le \be(t) \le b_2(\mu_1(t)) \ \ \mbox{and}\ \ \vk\si(\mu_1(t)) \le \rho(t) \le \si(\mu_1(t)).
$$
Therefore, 
\begin{equation}
\label{mainestimates}
b_1(Z(s)) \le \be(\tau(s)) \le b_2(Z(s)) \ \ \mbox{and}\ \ \vk\si(Z(s)) \le \rho(\tau(s)) \le \si(Z(s)).
\end{equation}
We have:
$$
\CB_1(Z(s)) \le \be(\tau(s))\frac{\sigma^2(Z(s))}{\rho^2(\tau(s))} \le \CB_2(Z(s)),
$$
where
$$
\CB_1(x) := \min(b_1(x), b_2(x), \vk^{-2}b_1(x), \vk^{-2}b_2(x)),\ \ 
\CB_2(x) := \max(b_1(x), b_2(x), \vk^{-2}b_1(x), \vk^{-2}b_2(x)).
$$
Recall that $0 < \vk \le 1$. Therefore, when $b_1(x) \ge 0$ (this is true, for example, for $x \in (0, x_1]$), we have: $\CB_1(x) = b_1(x)$ and $\CB_2(x) = \vk^{-2}b_2(x)$. When $b_2(x) \le 0$ (for example, for $x \in [x_2, 1 - \de)$), we have: $\CB_1(x) = \vk^{-2}b_1(x)$ and $\CB_2(x) = b_2(x)$. 

Let $Z_1(\cdot) = (Z_1(s), s \ge 0)$ and $Z_2(\cdot) = (Z_2(s), s \ge 0)$ be the solutions of the following stochastic differential equations:
$$
\d Z_1(s) = \CB_1(Z_1(s))\d s + \si^2(Z_1(s))\d B_0(s),\ \ \mbox{and}\ \ \d Z_2(s) = \CB_2(Z_2(s))\d s + \si^2(Z_2(s))\d B_0(s)
$$
with initial conditions $Z_1(0) = Z_2(0) = \mu_1(0)$. 
By Lemma~\ref{comparisonIW}, see Appendix, the process $Z(\cdot)$ a.s. does not hit $0$ if the process $Z_1(\cdot)$ a.s. does not hit $0$. Apply Feller's test (see Appendix): we want to show that
$$
\int_0^{x_1}\exp\biggl(-\int_{x_1}^y\frac{2\CB_1(z)}{\si^2(z)}dz\biggr)dy = \infty, 
$$
which can be equivalently written as
$$
\int_0^{x_1}\exp\biggl(\int_y^{x_1}\frac{2\CB_1(z)}{\si^2(z)}dz\biggr)dy = \infty.
$$
As mentioned above, for $x \in (0, x_1]$ we have: $b_1(x) \ge K_1 > 0$, and $\CB_1(x) = b_1(x)$. Therefore, we have:
\begin{align*}
\int_0^{x_1}&\exp\biggl(\int_y^{x_1}\frac{2\CB_1(z)}{\si^2(z)}dz\biggr)dy  \ge  \int_0^{x_1}\exp\biggl(K_1\int_y^{x_1}\frac{dz}{z^2(1 - z)^2}\biggr)dy \\ & \ge \int_0^{x_1}\exp\biggl(K_1\int_y^{x_1}\frac{dz}{z^2}\biggr)dy = \int_0^{x_1}\exp\biggl(K_1(1/y - 1/x_1)\biggr)dy \\ &= e^{-K_1/x_1}\int_0^{x_1}e^{K_1/y}dy  = \infty.
\end{align*}
This proves that $Z_1(\cdot)$ (as well as $Z(\cdot)$ and $\mu_1(\cdot)$) does not hit zero is complete.

\medskip

Now, let us deal with the other singularity: $1 - \de$. If the process $Z_1(\cdot)$ hits $1 - \de$ with positive probability, then the process $Z(\cdot)$  does the same. If the process $Z_2(\cdot)$ a.s. does not hit $1 - \de$, then the process $Z(\cdot)$ does the same.

The natural scale for the process $Z_1$ is
$$
\phi_1(x) := \int_{x_2}^x\exp\left(\int_{x_2}^y-\frac{2\CB_1(z)}{\si^2(z)}dz\right)dy.
$$
For $z \in [x_2, 1 - \de)$, we have: $\CB_1(z) = \vk^{-2}b_1(z)$, and
\begin{align*}
-\frac{2\CB_1(z)}{\si^2(z)} =& -\frac{2\vk^{-2}b_1(z)}{\si^2(z)} = -\frac{2(\psi(z)(1 - z) - C_1z)}{2\vk^2z^2(1 - z)^2} = -\frac{\psi(z)}{\vk^2z^2(1 - z)} + \frac{C_1}{\vk^2z(1 - z)^2}  \\ &= -\frac{z(-g(z) + 1/2) - z^2}{\vk^2z^2(1 - z)} + \frac{C_1}{\vk^2z(1 - z)^2} = A_1(z)g(z) + \eps_3(z),
\end{align*}
where we have defined
$$
\eps_3(z) := \frac{z - 1/2}{\vk^2z(1 - z)} + \frac{C_1}{\vk^2z(1 - z)^2}.
$$
Since the function $\eps_3$ is continuous on $[x_2, 1 - \de]$, it is bounded on this interval: there exists $C_3 > 0$ such that 
$|\eps_3(z)| \le C_3$ for $z \in [x_2, 1 - \de]$. 
Therefore,
$$
\phi_1(1 - \de) < \infty \ \ \mbox{if and only if}\ \ \int_{x_2}^{1 - \de}\exp\left(\int_{x_2}^yA_1(z)g(z)dz\right)dy < \infty.
$$
If these integrals are indeed finite, then 
$$
\int_{x_2}^{1 - \de}\frac{\phi_1(1 - \de) - \phi_1(x)}{\phi_1'(x)\si^2(x)}dx < \infty.
$$
This is verified in the same way as in the proof of Theorem~\ref{thm1}. Thus, if 
$$
\int_{x_2}^{1 - \de}\exp\left(\int_{x_2}^yA_1(z)g(z)dz\right)dy < \infty,
$$
then the process $Z_1(\cdot)$ hits $1 - \de$ with positive probability, and the market is not diverse. Moreover, we have: 
 $\phi_1(1 - \de) = \infty$ if and only if
$$
\int_{x_2}^{1 - \de}\exp\left(\int_{x_2}^yA_1(z)g(z)dz\right)dy = \infty.
$$
If $\phi_1(1 - \de) < \infty$, then the process $Z_1(\cdot)$ and therefore the processes $Z(\cdot)$ and $\mu_1(\cdot)$ hit $1 - \de$ with positive probability. The proof of part (i) is complete. 

Part (ii) is proved similarly. Indeed, in this case we have a scale
$$
\phi_2(x) := \int_{x_2}^x\exp\left(\int_{x_2}^y-\frac{2\CB_2(z)}{\si^2(z)}dz\right)dy.
$$
Let us show that $\phi_2(1 - \de) = \infty$. From here, it would follow that $Z_2(\cdot)$ a.s. does not hit $1 - \de$, and so $Z(\cdot)$ a.s. does not hit $1 - \de$, and the market is $\de$-diverse. Let
$$
\ol{\phi}_2(x) := \int_{x_2}^x\exp\left(\int_{x_2}^y-\frac{2\ol{b}_2(z)}{\si^2(z)}dz\right)dy.
$$
As mentioned above, for $z \in [x_2, 1 - \de)$, we have: $b_2(z) \le 0$, and so $\CB_2(z) = b_2(z) \le \ol{b}_2(z)$, and $\ol{\phi}(x) \le \phi(x)$ for $x \in [1/2, 1 - \de)$. Therefore, it suffices to show that $\ol{\phi}_2(1 - \de) = \infty$. The proof of this statement is similar to that from part (i). $\blacksquare$

The proof of Corollary~\ref{cor2} is similar to the proof of Corollary~\ref{cor1} and is left to the reader.

\section{Appendix}

Let us state Feller's test for explosions. It is taken from \cite[Section 6.2]{DurBook}. Fix an interval $(\al, \be) \subseteq \BR$ and a point $x_0 \in (\al, \be)$. Let $X(\cdot) = (X(t), t \ge 0)$ be a real-valued stochastic process satisfying the following stochastic differential equation:
$$
\d X(t) = b(X(t))\d t + \si(X(t))\d B(t),\ \ t \ge 0,\ \ X(0) = x_0.
$$
Here, $B = (B(t), t \ge 0)$ is a one-dimensional standard Brownian motion, and $b, \si : (\al, \be) \to \BR$ are continuous functions such that $\si^2(x) > 0$ for all $x \in (\al, \be)$. Define the {\it natural scale} as 
$$
\phi(x) := \int_{x_0}^x\exp\left[\int_{x_0}^y-\frac{2b(z)}{\si^2(z)}dz\right]dy,\ \ x \in (\al, \be).
$$
Also, let
$$
m(x) = \frac1{\phi'(x)\si^2(x)},\ \ x \in (\al, \be).
$$
Finally, define the following integrals:
$$
I_{\al} := \int_{\al}^{x_0}(\phi(x) - \phi(\al))m(x)\, dx,\ \ I_{\be} = \int_{x_0}^{\be}(\phi(\be) - \phi(x))m(x)\, dx.
$$

\begin{proposition} (i) If $\phi(\al) = -\infty$, or $\phi(\al) > -\infty$ and $I_{\al} = \infty$, then the process $X(\cdot)$ a.s. does not hit $\al$. Otherwise, it does hit $\al$ with positive probability.

(ii) If $\phi(\be) = \infty$, or $\phi(\be) < \infty$ but $I_{\be} = \infty$, then the process $X(\cdot)$ a.s. does not hit $\be$. Otherwise, it does hit $\be$ with positive probability.
\end{proposition}

{\it Proof of Lemma~\ref{majorlemma}.} From~\eqref{main}, we have: 
$$
\d X_i(t) = X_i(t)\left[\left(-g(\mu_i(t)) + \frac12\right)\d t + \d W_i(t)\right],\ \ i = 1, \ldots, n 
$$
and
$$
\d S(t) = \SL_{i=1}^nX_i(t)\left[\left(-g(\mu_i(t)) + \frac12\right)\d t + \d W_i(t)\right].
$$
Therefore,
$$
d<X_i, S>_t = X_i^2(t)\d t,\ \ d<S>_t = \SL_{i=1}^nX_i^2(t)\, \d t.
$$
Apply Ito's formula to $X_i(t)/S(t)$. Let $f(x, y) = x/y$, then 
$$
f_x = \frac1y,\ f_y = -\frac{x}{y^2},\ f_{xx} = 0,\ f_{xy} = -\frac1{y^2},\ f_{yy} = \frac{2x}{y^3}.
$$
Therefore, we have:
\begin{align*}
\d\mu_i(t) &= \d f(X_i(t), S(t)) = f_x(X_i(t), S(t))\d X_i(t) + f_y(X_i(t), S(t))\d S(t)  \\ &+
\frac12f_{xx}(X_i(t), S(t))\d<X_i>_t + f_{xy}(X_i(t), S(t))\d<X_i, S>_t + \frac12f_{yy}(X_i(t), S(t))\d<S>_t \\ & = 
\frac1{S(t)}X_i(t)\left[\left(-g(\mu_i(t)) + \frac12\right)\d t + \d W_i(t)\right] - \frac{X_i(t)}{S^2(t)}\SL_{j=1}^nX_j(t)\left[\left(-g(\mu_j(t)) + \frac12\right)\d t + \d W_j(t)\right]  \\   & \hspace{5cm}  - 
\frac1{S^2(t)}X_i^2(t)\d t + \frac{X_i(t)}{S^3(t)}\SL_{j=1}^nX_j^2(t)\d t.
\end{align*}
We can express this in terms of market weights as
\begin{align*}
\mu_i(t)&\left(-g(\mu_i(t)) + \frac12\right)\d t + \mu_i(t)\d W_i(t) - \mu_i(t)\SL_{j=1}^n\mu_j(t)\left(-g(\mu_j(t)) + \frac12\right)\d t 
\\ & - \mu_i(t)\SL_{j=1}^n\mu_j(t)\d W_j(t)  - \mu_i^2(t)\d t + \mu_i(t)\SL_{j=1}^n\mu_j^2(t)\d t \\ &
= \biggl[\psi(\mu_i(t)) - \mu_i(t)\SL_{j=1}^n\psi(\mu_j(t))\biggr]\d t + \SL_{j=1}^n(\de_{ij}\mu_i - \mu_i\mu_j)\d W_j(t). \blacksquare
\end{align*}

\begin{lemma} Assume $X = (X_t, t \ge 0)$ is a progressively measurable continuous stochastic process such that 
$$
X_t = x + \int_0^t\ga_u\d u + \int_0^t\rho_u\d W_u,
$$
where $B$ is an $(\CF_t)_{t \ge 0}$-Brownian motion, and the processes $\ga = (\ga_t)_{t \ge 0}$ and $\rho = (\rho_t)_{t \ge 0}$ are progressively measurable. Assume that, for some constants $0 < \vk_1 \le \vk_2$, we have the following estimate:
$$
0 < \vk_1 \le \frac{|\rho_t|}{\si(X_t)} \le \vk_2 < \infty,\ \ t \ge 0,
$$
where $\si : \BR \to (0, \infty)$ is a continuous real-valued function. Then the following time-change process
$$
\Delta(t) := \int_0^t\frac{\rho^2_s}{\si^2(X_s)}ds.
$$
is a strictly increasing function, $\Delta(0) = 0,\ \Delta(\infty) = \infty$. Define its inverse: 
$$
\tau(s) := \inf\{t \ge 0 \mid \Delta(t) \ge s\}.
$$
The family $(\CF_{\tau(s)})_{s \ge 0}$ of $\sigma$-algebras is a filtration satisfying the usual conditions. Moreover, the process $Z = (Z_s = X_{\tau(s)})_{s \ge 0}$ satisfies the equation
$$
\d Z_s = \ga_{\tau(s)}\frac{\si^2(Z_s)}{\rho^2_{\tau(s)}}\d s + \si(Z_s)\d B_s,
$$
where $B = (B_s)_{s \ge 0}$ is another $(\CF_{\tau(s)})_{s \ge 0}$-Brownian motion. 
\label{Hajek}
\end{lemma}

\begin{proof} Analogous to \cite[Section 3]{H1985}. For $s \ge 0$, we have: 
\begin{equation}
\label{abc}
Z_s = X_{\tau(s)} = x + \int_0^{\tau(s)}\ga_u\d u + \int_0^{\tau(s)}\rho_u\d W_u.
\end{equation}
Let us change variables in the first integral from the right-hand side of~\eqref{abc}: $u = \tau(v),\ v = \De(u),\ 0 \le v \le s$. Then
$$
\d u = \tau'(v)\d v = \frac1{\De'(\tau(v))}\d v = 
\frac{\si^2(X_{\tau(v)})}{\rho^2_{\tau(v)}}\d v = 
\frac{\si^2(Z_v)}{\rho^2_{\tau(v)}}\d v.
$$
Therefore,
$$
\int_0^{\tau(s)}\ga_u\d u = \int_0^s\ga_{\tau(v)}\frac{\si^2(Z_v)}{\rho^2_{\tau(v)}}\d v.
$$
Now, consider the It\^o integral from the right-hand side of the~\eqref{abc}: let 
$$
M_t = \int_0^t\frac{\rho_u}{\si(X_u)}\d W_u,\ \ t \ge 0.
$$
Then $M = (M_t)_{t \ge 0}$ is a continuous $(\CF_t)_{t \ge 0}$-martingale with $<M>_t = \De(t)$, and
$$
\int_0^t\rho_u\d W_u = \int_0^t\si(X_u)\d M_u.
$$
By~\cite[Theorem 3.4.6]{KS1991}, $M_t = B_{\De(t)}$ for $t \ge 0$, where
$B = (B_s)_{s \ge 0}$ defined by $B_s = M_{\tau(s)},\ s \ge 0$ is another $(\CF_{\tau(s)})_{s \ge 0}$-Brownian motion, and this filtration satisfies the usual conditions. 
By \cite[Proposition 3.4.8]{KS1991}, 
$$
\int_0^t\si(X_u)\d M_u = \int_0^{\De(t)}\si(X_{\tau(v)})\d B_v = \int_0^{\De(t)}\si(Z_v)\d B_v.
$$
Since $\De(\tau(s)) = s$, we have:
$$
\int_0^{\tau(s)}\rho_u\d W_u = \int_0^{\tau(s)}\si(X_u)\d M_u = \int_0^{s}\si(Z_v)\d B_v.
$$
Thus, for $s \ge 0$ we have:
$$
Z_s = x + \int_0^s\ga_{\tau(v)}\frac{\si^2(Z_v)}{\rho^2_{\tau(v)}}\d v + 
\int_0^{s}\si(Z_v)\d B_v,
$$
which completes the proof.
\end{proof}

\begin{lemma} Assume $X = (X_t)_{t \ge 0}$ and $Y = (Y_t)_{t \ge 0}$ are two progressively measurable continuous stochastic processes which satisfy
$$
\d X_t = \be_t\d t + \si(X_t)\d W_t,\ X_0 = x;\ \ \d Y_t = b(Y_t)\d t + \si(Y_t)\d W_t,\ Y_0 = x.
$$
Here, $\be = (\be_t)_{t \ge 0}$ is a progressively measurable process, and $b,\ \si : \BR \to \BR$ are real-valued continuous functions. 
If $\be_t \le b(X_t)$ a.s. for $t \ge 0$, then $X_t \le Y_t$ a.s. for $t \ge 0$. 
If $\be_t \ge b(X_t)$ a.s. for $t \ge 0$, then $X_t \ge Y_t$ a.s. for $t \ge 0$. 
\label{comparisonIW}
\end{lemma}

\begin{proof} Follows from \cite[Lemma 6.1]{IWBook}.
\end{proof}

\begin{acknowledgements}

I would like to thank Professor Tomoyuki Ichiba for suggesting this problem. I would also like to thank my adviser, Professor Soumik Pal, for help and encouragement. I am thankful to an anonymous referee for reading this article very carefully, and making lots of useful comments. Last but not least, I am indebted to  Professor Ioannis Karatzas for useful discussion and valuable comments which, in particular, helped clarify the proof of Theorem~\ref{thm2}.  

\end{acknowledgements}

\bibliographystyle{plain}

\bibliography{aggregated}

\begin{thebibliography}{10}

\bibitem{FFR2007}
Francesco Audrino, Robert Fernholz, and Roberto~G. Ferretti.
\newblock A forecasting model for stock market diversity.
\newblock {\em Annals of Finance}, 3(2):213--240, 2007.

\bibitem{BF2008}
Adrian~D. Banner and Daniel Fernholz.
\newblock Short-term relative arbitrage in volatility-stabilized markets.
\newblock {\em Ann. Finance}, 4(4):445 -- 454, 2008.

\bibitem{BFK2005}
Adrian~D. Banner, Robert Fernholz, and Ioannis Karatzas.
\newblock Atlas models of equity markets.
\newblock {\em Ann. Appl. Probab.}, 15(4):2296--2330, 2005.

\bibitem{CP2010}
Sourav Chatterjee and Soumik Pal.
\newblock A phase transition behavior for {B}rownian motions interacting
  through their ranks.
\newblock {\em Probab. Theory Related Fields}, 147(1-2):123--159, 2010.

\bibitem{DurBook}
Richard Durrett.
\newblock {\em Stochastic calculus}.
\newblock Probability and Stochastics Series. CRC Press, Boca Raton, FL, 1996.
\newblock A practical introduction.

\bibitem{F2002}
Robert Fernholz.
\newblock {\em Stochastic portfolio theory}, volume~48 of {\em Applications of
  Mathematics (New York)}.
\newblock Springer-Verlag, New York, 2002.
\newblock Stochastic Modelling and Applied Probability.

\bibitem{FIK2013b}
Robert Fernholz, Tomoyuki Ichiba, and Ioannis Karatzas.
\newblock A second-order stock market model.
\newblock {\em Annals of Finance}, 9(3):439--454, 2013.

\bibitem{FIK2013}
Robert Fernholz, Tomoyuki Ichiba, and Ioannis Karatzas.
\newblock Two {B}rownian particles with rank-based characteristics and
  skew-elastic collisions.
\newblock {\em Stochastic Process. Appl.}, 123(8):2999--3026, 2013.

\bibitem{FK2005}
Robert Fernholz and Ioannis Karatzas.
\newblock Relative arbitrage in volatility-stabilized markets.
\newblock {\em Annals of Finance}, 1:149--177, 2005.

\bibitem{FK2009}
Robert Fernholz and Ioannis Karatzas.
\newblock Stochastic portfolio theory: an overview.
\newblock {\em Handbook of Numerical Analysis}, 15:89--167, 2009.

\bibitem{FKK2005}
Robert Fernholz, Ioannis Karatzas, and Constantinos Kardaras.
\newblock Diversity and relative arbitrage in equity markets.
\newblock {\em Finance Stoch.}, 9(1):1--27, 2005.

\bibitem{H1985}
Bruce Hajek.
\newblock Mean stochastic comparison of diffusions.
\newblock {\em Z. Wahrsch. Verw. Gebiete}, 68(3):315--329, 1985.

\bibitem{IchibaThesis}
Tomoyuki Ichiba.
\newblock {\em Topics in multi-dimensional diffusion theory: {A}ttainability,
  reflection, ergodicity and rankings}.
\newblock ProQuest LLC, Ann Arbor, MI, 2009.
\newblock Thesis (Ph.D.)--Columbia University.

\bibitem{IKS2013}
Tomoyuki Ichiba, Ioannis Karatzas, and Mykhaylo Shkolnikov.
\newblock Strong solutions of stochastic equations with rank-based
  coefficients.
\newblock {\em Probab. Theory Related Fields}, 156(1-2):229--248, 2013.

\bibitem{Ichiba11}
Tomoyuki Ichiba, Vassilios Papathanakos, Adrian Banner, Ioannis Karatzas, and
  Robert Fernholz.
\newblock Hybrid atlas models.
\newblock {\em Ann. Appl. Probab.}, 21(2):609--644, 2011.

\bibitem{IWBook}
Nobuyuki Ikeda and Shinzo Watanabe.
\newblock {\em Stochastic differential equations and diffusion processes},
  volume~24 of {\em North-Holland Mathematical Library}.
\newblock North-Holland Publishing Co., Amsterdam, second edition, 1989.

\bibitem{KPS2012}
Ioannis Karatzas, Soumik Pal, and Mykhaylo Shkolnikov.
\newblock Systems of brownian particles with asymmetric collisions.
\newblock 2012.
\newblock Preprint. Available at arXiv:1210.0259v1.

\bibitem{KS1991}
Ioannis Karatzas and Steven~E. Shreve.
\newblock {\em Brownian motion and stochastic calculus}, volume 113 of {\em
  Graduate Texts in Mathematics}.
\newblock Springer-Verlag, New York, second edition, 1991.

\bibitem{K2008}
Constantinos Kardaras.
\newblock Balance, growth and diversity of financial markets.
\newblock {\em Annals of Finance}, 4(3):369--397, 2008.

\bibitem{OR2006}
Jorg~R. Osterrieder and Thorsten Rheinlander.
\newblock Arbitrage opportunities in diverse markets via a non-equivalent
  measure change.
\newblock {\em Annals of Finance}, 2:287--301, 2006.

\bibitem{P2011}
Soumik Pal.
\newblock Analysis of market weights under volatility-stabilized market models.
\newblock {\em Ann. Appl. Probab.}, 21(3):1180--1213, 2011.

\bibitem{PP2008}
Soumik Pal and Jim Pitman.
\newblock One-dimensional {B}rownian particle systems with rank-dependent
  drifts.
\newblock {\em Ann. Appl. Probab.}, 18(6):2179--2207, 2008.

\bibitem{PS2010}
Soumik Pal and Mykhaylo Shkolnikov.
\newblock Concentration of measure for systems of brownian particles
  interacting through their ranks.
\newblock 2010.
\newblock Preprint. Available at arXiv:1011.2443.

\bibitem{P2013}
Radka Pickova.
\newblock Generalized volatility-stabilized processes.
\newblock {\em Annals of Finance}, pages 1--25, 2013.

\bibitem{S2011}
Mykhaylo Shkolnikov.
\newblock Competing particle systems evolving by interacting {L}\'evy
  processes.
\newblock {\em Ann. Appl. Probab.}, 21(5):1911--1932, 2011.

\bibitem{S2013}
Mykhaylo Shkolnikov.
\newblock Large volatility-stabilized markets.
\newblock {\em Stochastic Process. Appl.}, 123(1):212--228, 2013.

\bibitem{FS2011}
Winslow Strong and Jean-Pierre Fouque.
\newblock Diversity and arbitrage in a regulatory breakup model.
\newblock {\em Annals of Finance}, 7:349--374, 2011.

\end{thebibliography}

\end{document}